\documentclass[oneside,12pt]{article}
\usepackage[english]{babel}
\usepackage{amssymb}
\usepackage{latexsym}
\usepackage{graphics}
\usepackage{graphicx}
\RequirePackage{latexsym}
\RequirePackage{amsthm}
\RequirePackage{amsmath}
\RequirePackage{amssymb}
\RequirePackage{makeidx}
\usepackage{amsthm}

\usepackage{fancyhdr}
\fancypagestyle{plain}{%
	\fancyhead{} 
}

\usepackage[utf8x]{inputenc}
\usepackage{latexsym}
\usepackage{amssymb}
\usepackage{amsmath}
\usepackage{amsfonts}
\usepackage{amsthm}
\usepackage{fontenc}
\usepackage{paralist}
\usepackage{enumerate}
\usepackage{mathrsfs}
\usepackage[retainorgcmds]{IEEEtrantools}
\usepackage[usenames,dvipsnames]{color} 
\usepackage{float} 

\usepackage{graphicx}
\theoremstyle{definition}
\newtheorem{defin}{Definition}[section]
\newtheorem{prop}[defin]{Proposition}

\newtheorem{theorem}[defin]{Theorem}

\newtheorem{lemma}[defin]{Lemma}
\newtheorem{cor}[defin]{Corollary}

\newtheorem*{remark}{Remark}

\theoremstyle{plain} \newtheorem{exa}[defin]{Example}

\usepackage{graphicx}

\usepackage[colorinlistoftodos, textwidth=2.6cm, obeyDraft]{todonotes} 
\usepackage[T1]{fontenc}
\usepackage{epigraph}

\makeindex

\begin{document}
\title{{Irreducible laminations for IWIP Automorphisms of a free product and Centralisers}}	
	
\author{Dionysios Syrigos}
\date{}
\maketitle

	\begin{abstract}
	For every free product decomposition $G = G_{1} \ast ...\ast G_{q} \ast F_{r}$, where $F_r$ is a finitely generated free group, of a group $G$ of finite Kurosh rank, we can associate some (relative) outer space $\mathcal{O}$. In this paper, we develop the theory of (stable) laminations for (relative) irreducible with irreducible powers (IWIP) automorphisms. In particular, we examine the action of $Out(G, \mathcal{O}) \leq Out(G)$ (i.e. the automorphisms which preserve the set of conjugacy classes of $G_i$'s) on the set of laminations. We generalise the theory of the attractive laminations associated to automorphisms of finitely generated free groups. The strategy is the same as in the classical case (see \cite{BFM}), but some statements are slightly different because of the existence of the $G_i$'s. More precisely, we prove that the stabiliser of the lamination of a relative IWIP is a $\mathbb{Z}$-extension of a subgroup that is consisted of virtually elliptic automorphisms. Note that in the free case, virtually elliptic automorphisms are exactly the finite order automorphisms of $Out(F_n)$.\\
	As a corollary of the previous theorem, we generalise the fact that the centraliser of an IWIP automorphism of a free group, is virtually cyclic. As a direct corollary, if $Out(G)$ is virtually torsion free and every $Aut(G_i)$ is finite, we prove that the centraliser of an IWIP is virtually cyclic. Finally, we give an example which shows that we cannot expect that in general the centraliser of a relative IWIP (and as a consequence the stabiliser of its stable lamination) is virtually cyclic.
	
	\end{abstract}

		\newpage
	\tableofcontents
	\newpage
\section{Introduction}
\par
Let $G$ be a group which splits as a free product $G = G_{1} \ast ...\ast G_{q} \ast F_{r}$. Guirardel and Levitt in \cite{GF} constructed an outer space relative to any free product decomposition for a f.g. group. Let $Out(G, \mathcal{O})$ be the subgroup of $Out(G)$, which consists of the automorphisms which preserve the conjugacy classes of $G_i$'s (note that in the case of the Grushko decomposition, $Out(G) = Out(G, \mathcal{O})$).
We could define the notion of irreducibility using representatives of automorphisms which leave invariant subgraphs, but here it is less complicated to use the notion of free factor systems. 
More specifically, we say that an element $ \phi \in Out(G, \mathcal{O}) $ is irreducible relative to $\mathcal{O}$, if the corresponding free factor system $\mathcal{G} = \{ [G_i] : 1 \leq i  \leq q \}$ is a maximal proper, $\phi$-invariant free factor system. Therefore we define the notion of an irreducible with irreducible powers (or simply IWIP) automorphism relative to $\mathcal{O}$, as in the special case where $G$ is a finitely generated free group.\par
In this paper, we study IWIP automorphisms and in particular we show that we can define the stable (and unstable) lamination $\Lambda$ associated to an IWIP, using exactly the same method as in the free case. In the classical case, it can be proved that the stabiliser of the lamination is virtually cyclic (see \cite{BFM}).
However, in the general case, the presence of the $G_i$'s, does not allow us to get the same statement and as we will see later this is not true in general, but can prove the following generalisation (note that if $\phi \in Out(F_n)$ fixes a point of $CV_n$, it has finite order):

\begin{theorem}
Let $\phi$ be an IWIP relative to some relative outer space $\mathcal{O}$.
Let's denote by $Stab( \Lambda ) $ the $Out(G,\mathcal{O})$- stabiliser of the stable lamination $\Lambda$.
Then $Stab(\Lambda)$ is an infinite cyclic extension of a subgroup $A$, where every $\phi \in A$ virtually fixes a point of $\mathcal{O}$.
\end{theorem}
Note that the stabilisers on any point of $\mathcal{O}$ is well known by a theorem of \cite{GF} and it depends on the valence of non-free vertices of the points and on the $G_i$'s.\\
As a consequence of our proof, we can get the following theorems for some special cases:
\begin{theorem}
If we suppose that every $Inn(G_i)$ is finite, then we have that:
	\begin{enumerate}
		\item There is a normal periodic subgroup $A$ of $Stab(\Lambda)$, such that the group $Stab(\Lambda) / A$ has a normal subgroup $B$ isomorphic to subgroup of $\bigoplus \limits_{i=1} ^{q} Out(G_i)$ and $(Stab(\Lambda) / A) / B$ is isomorphic to $\mathbb{Z}$.
		\item Let's suppose that $Out(G)$ is virtually torsion free. Then $Stab(\Lambda)$ has a (torsion free) finite index subgroup $K$ such that $K / B'$ is isomorphic to $\mathbb{Z}$, where $B'$ is a normal subgroup of $K$ isomorphic to subgroup of $\bigoplus \limits _{i=1} ^{q} Out(G_i)$.
	\end{enumerate}
\end{theorem}
Similarly, it can be proved the following:
\begin{theorem}
	If every $Aut(G_i)$ is finite and $Out(G)$ is virtually torsion free, then $Stab(\Lambda)$ is virtually infinite cyclic.
\end{theorem}
As a further motivation to study the laminations for this general case is that we can find the notion of laminations for a free group in a lot of different forms and contexts in the literature and study of them implies important results (for example, see \cite{BFM} \cite{BFMH1}, \cite{BFMH2}, \cite{CHL1}, \cite{CHL2} and \cite{HL}), therefore it looks like interesting to generalise this notion in a more general context. In addition, further motivation is that we can find natural generalisations for a lot of facts about $CV_n$ in the general case, for example in $\cite{FM}$, Francaviglia and Martino  generalised a lot of tools like train track maps and the Lipschitz metric. But there are some recent papers that show we can also use further methods of studying $Out(Fn)$ for $Out(G)$ (where $G$ is written as free product as above) such that the closure of outer space, the Tits alternatives for $Out(G)$, the hyperbolic complex corresponding to $Out(G)$ and the asymmetry of the outer space of a free product (see \cite{H1}, \cite{H2}, \cite{H3}, \cite{K2} and \cite{S1}). Finally, the author as an application of the results of the present paper generalises a result of Hilion (\cite{Hi}), about the stabiliser of attractive fixed point of an IWIP automorphism (\cite{S2}).\par

Given a group $G$ and an element $g \in G,$ a natural question is to study the centraliser $C(g)$ of $g$ in $G$. In several classes of groups, centralisers of elements are reasonably well-understood and sometimes they are useful to the study of the group. For example, Feighn and Handel in \cite{FH} classified abelian subgroups in $Out(F_n)$ by studying centralisers of elements.
Moreover, a well known result for an IWIP automorphism of a free groups (there are several proofs, see \cite{BFM}, \cite{L} or \cite{KL1}) states that their centralisers are virtually cyclic. Again, it's not true for a relative IWIP, but in the general case we can get similar statements like the theorems above and namely:
\begin{theorem}
	Let $\phi$ be an IWIP as above. Let's denote by $C(\phi)$ the centraliser of $\phi$ in $Out(G, \mathcal{O})$.
	Then there is normal subgroup $A$ of $C(\phi)$, where $C(\phi)$ a cyclic extension of $A$. Moreover, every $\phi \in A$ virtually fixes a point of $\mathcal{O}$.
\end{theorem}

Moreover, in the special case that the inner automorphisms groups of the factors are finite, we get:

 \begin{theorem}
 	If we suppose that every $Inn(G_i)$ is finite, then we can get the following:
 	\begin{enumerate}
 		
 		\item There is a normal periodic subgroup $A_1$ of $C(\phi)$, such that the group $C(\phi) / A_1$ has a normal subgroup $B_1$ isomorphic to subgroup of $\bigoplus \limits_{i=1} ^{q} Out(G_i)$ and $(C(\phi) / A_1) / B_1$ is isomorphic to $\mathbb{Z}$.
 		\item Let's also suppose that $Out(G)$ is virtually torsion free. Then $C(\phi)$ has a (torsion free) finite index subgroup $A_1 '$ such that $A_1 ' / B_1$ is isomorphic to $\mathbb{Z}$, where $B_1$ is a normal subgroup of $A_1$ isomorphic to subgroup of $\bigoplus \limits _{i=1} ^{q} Out(G_i)$.

 	\end{enumerate}
 \end{theorem}
 Also, the following theorem holds (and it is a direct generalisation of the theorem of the free case):
 \begin{theorem}
 	If every $Aut(G_i)$ is finite and $Out(G)$ is virtually torsion free, then $Stab(\Lambda)$ is virtually infinite cyclic.
 \end{theorem}
 
In fact, we can get stronger results for commensurators of the automorphisms instead of centralisers.\\
Note that there are a lot of IWIP automorphisms that don't commute with the factor automorphisms of $G_i$'s and in particular for them, the theorem above implies that their centralisers are virtually cyclic.\\
On the other hand, there are examples of IWIP automorphisms which have big centraliser (in particular, they are not virtually cyclic).\\
\begin{exa}
 We fix the free product decomposition $G = G_1 \ast <b_1> \ast <b_2>$, where $b_i$ are of infinite order and we denote by $F_2 = <b_1> \ast <b_2>$ the "free part". Then in the corresponding outer space $\mathcal{O}(G, G_1, F_2)$, which we denote by $\mathcal{O}$. In each tree $T \in \mathcal{O}$ there is exactly one non free vertex $v_1$ s.t $G_{v_1} = G_1$. Then we define the outer automorphism $\phi$, which satisfies $\phi(a) = a$ for every $a \in G_1$, $\phi(b_1)= b_2g_1, \phi(b_2) = b_1 b_2$ for some $g_1 \in G_1$, then we can see that $\phi \in Out(G, \mathcal{O})$ is an IWIP relative to $\mathcal{O}$. But then every factor automorphism of $G_1$ that fixes $g_1$ commutes with $\phi$ and therefore $C(\phi)$ contains the subgroup $A$ of  $Aut(G_1)Inn(G)$ that fixes $g_1$. So if $A$ is sufficiently big, the relative centraliser is not virtually cyclic, but the quotient modulo its intersection with $A$ is (by applying the previous theorem).
 Since we can change $G_1$ with any group (of finite Kurosh rank) and we can get automorphisms with arbitarily big centralisers. For example, if $G_1$ is isomorphic to $F_3$ and $g_1$ an element of its free basis, we have that $C(\phi)$ contains a subgroup which is isomorphic to $Aut(F_2)Inn(G)$.
\end{exa}

\textbf{Strategy of the proof}:
The paper is organized as follows:\par
In Section 2, we recall some preliminary definitions, facts and well known results  about  the outer space of a free product. In Section 3, we prove a useful technical lemma for $\mathcal{O}$-maps, more specifically we prove that every two such maps are equal except possibly two bounded (depends only on the map, not the path) paths near the endpoints. 
The next sections form the main part of this paper and we follow exactly the same approach as in \cite{BFM}.
In section 4, we define the lamination using train track representatives, and then we extend the notion to any tree. Also, we list some useful properties.
In Section 5, we define the action of $Out(G,\mathcal{O})$ on the set of irreducible laminations. In Section 6 we define the notion of a subgroup which carries the lamination and then we prove that any such subgroup has finite index in the whole group.
In Section 7, which is the most  crucial for our arguments, we construct  a homomorphism from the stabiliser to group of positive real numbers. Then in Section 8, we study the kernel of this homomorphism, and in particular, we  prove that any element of the kernel is non-exponentially growing and in the reducible case it has a relative train track representative with a very good form restricted to the lower strata.
Also, we prove the discreteness of the image which allows us to think the previous map, as a homomorphism from the stabiliser to the group of integers. Finally, in
Section 9, we prove some useful lemmas and the main results.\\
\textbf{Acknowledgements.} I warmly thank my advisor Armando Martino for his help, suggestions, ideas, and corrections.\\
I would also like to thank Lee Mosher for pointing me out the inconsistency of a statement of the first version of this paper and Ashot Minasyan for his suggestion to generalise the main result using commensurators instead of centralisers.
Finally, I would like to thank the anonymous referees that pointed me out that the kernel of the action is trivial and for providing me a counter-example for a previous statement of my results.

\section{Preliminaries}

\subsection{Outer space and $\mathcal{O}$-maps}
In this subsection we recall the definitions of outer space and some basic properties. For example, the existence of $\mathcal{O}$ - maps between any two elements of the space which is a very useful tool.\\
Everything in the present and the next subsection about the outer space, the $\mathcal{O}$ - maps and the train track representatives can be found in \cite{FM}.\\  
Let $G$ be a group which splits as a finite free product of the following form $G = H_{1} \ast ...\ast H_{q} \ast F_{r}$, where every $H_i$ is non-trivial, not isomorphic to $\mathbb{Z}$ and freely indecomposable. We say that such a group has \textit{ finite Kurosh rank} and such a decomposition is called \textit{Gruskho decomposition}. For example, every f.g. group admits a splitting as above (by the Grushko's theorem). We are interested only for groups which have finite Kurosh rank.
\\
Now for a group $G$, as above, we fix an arbitary (non-trivial) free product decomposition $G = H_{1} \ast ...\ast H_{q} \ast F_{r}$ (without the assumption that the $H_i$'s are not isomorphic to $\mathbb{Z}$ or freely indecomposable), but we additionally suppose that $r>0$. These groups admit co-compact actions on $\mathbb{R}$-trees (and vice-versa). It is useful that we can also apply the theory in the case that $G$ is free, and the $G_i$'s are certain free factors of $G$ (relative free case).\\
We will define an outer space $\mathcal{O} = \mathcal{O}(G, (G_i)^{p}_{i=1}, F_r)$ relative to the free product decomposition (or relative outer space). The elements of the outer space can be thought as simplicial metric $G$-trees, up to $G$-equivariant homothety.
Moreover, we require that these trees also satisfy the following:
\begin{itemize}
\item The action of $G$ on $T$ is minimal.
\item The edge stabilisers are trivial.
\item There are finitely many orbits of vertices with non-trivial stabiliser, more precisely for every $H_i$, $i = 1,..., q$ (as above) there is exactly one vertex $v_i$ with stabiliser $H_i$ (all the vertices in the orbits of $v_i$'s are called \textit{non-free vertices}).
\item All other vertices have trivial stabiliser (and we call them \textit{free vertices}).
\end{itemize}
The quotient $G / T$ is a finite graph of groups. We could also define the outer space as the space of "marked metric graph of groups" using the quotients instead of the trees, but we won't use this point of view because here it is easier to work using the trees. However, we use the quotients when the statements in this context are less complicated.\\
We would like to define a natural action of $Out(G)$ on $\mathcal{O}$, but this is not possible since it not always the case that the automorphisms preserve the structure of the trees (i.e. they don't send non-free vertices to non-free vertices). However, we can describe here the action of a specific subgroup of $Out(G)$ (namely, the automorphisms that preserve the decomposition or equivalently the structure of the trees) on $\mathcal{O}$.\\
Let $Aut(G, \mathcal{O})$ be the subgroup of $Aut(G)$  that preserve the set of conjugacy classes of the $G_i$ 's. Equivalently, $\phi \in Aut(G)$ belongs to $Aut(G,\mathcal{O})$ iff $\phi(G_i)$ is conjugate to one of the $G_j$ 's. The group $Aut(G,\mathcal{O})$ admits a natural action on a simplicial tree by "changing the action", i.e. for $\phi \in Aut(G, \mathcal{O})$ and $T \in \mathcal{O}$, we define $\phi(T)$ to be the element with the same underlying tree with $T$, the same metric but the action is given by $g*x = \phi(g)x$ (where the action in the right hand side is the action of the $G$-tree $T$). Now since the set of inner automorphisms of $G$, $Inn(G)$ acts trivially on $\mathcal{O}$ we can define $Out (G,\mathcal{O}) = Aut(G,\mathcal{O})/ Inn(G)$ which acts on $\mathcal{O}$ as above. Note that in the case of the Grushko decomposition we have $Out(G) = Out(G,\mathcal{O})$.\\
We say that a map between trees $A,B \in \mathcal{O}$, $f : A \rightarrow B$ is an \textbf{$\mathcal{O}$- map}, if it is a $G$-equivariant, Lipschitz continuous, surjective function.
Note here that we denote by $Lip(f)$ the Lipschitz constant of $f$.
\\
It is very useful to know that there are such maps between any two trees. This is true and, additionally, by their construction they coincide on the non - free vertices (and in section 3, we prove that every two such maps "almost" coincide). More specifically, by \cite{FM}, we get:
\begin{lemma}
For every pair $A,B \in \mathcal{O}$; there exists a $\mathcal{O}$-map $f : A \rightarrow B$.
Moreover, any two $\mathcal{O}$-maps from $A$ to $B$ coincide on the non-free vertices.
\end{lemma}

Let $f: A \rightarrow A$ be a simplicial (sending vertices to vertices and edges to edge-paths) $\mathcal{O}$-map, where $A \in \mathcal{O}$. Then $f$ induces a map (here we denote by $Df$ the map which sends every edge $e$ to the first edge of the edge path $f(e)$) on the set of turns, sending every turn $(e_1, e_2)$ to the turn $(Df(e_1),Df(e_2))$. Then as usually, we say that the turn $(e_1, e_2)$ is \textit{legal}, if for every $k$ the turn $(Df^{k}(e_1), Df^{k}(e_2))$ is non-degenerate. This induces a pre-train track structure on the set of edges at each vertex. But there are also different pre-train track structures and one of which we will use later, therefore we need the general definition.
\begin{defin}
\begin{enumerate}
\item A \textbf{pre-train track structure} on a $G$-tree $T$ is a $G$-invariant equivalence relation on the set of germs of edges at each vertex of $T$.
Equivalence classes of germs are called \textbf{gates}.
\item A \textbf{train track structure} on a $G$-tree T is a pre-train track structure with at least two gates at every vertex.
\item A \textbf{turn} is a pair of germs of edges emanating from the same vertex. A \textbf{legal turn} is called a turn for which the two germs belong to different equivalent classes. A \textbf{legal path}, is a path that contains only legal turns.
\end{enumerate}	
\end{defin}

A  pre-train track structure induced by some $\mathcal{O}$ - map is not always a train track structure, but there are some $\mathcal{O}$ - maps (we call them optimal maps) which induce train track structures. But firstly we need the notion of PL maps (which corresponds to piecewise linear homotopy equivalence in the free case). We call a map between two elements of the outer space \textbf{PL}, if it is piecewise linear and $\mathcal{O}$-map.
We denote by $A_{max}(f)$ the subgraph of $A$ consisting on those edges $e$ of $A$ for which $S_{f,e} = Lip(f)$ (i.e. the set of edges which are maximally stretched by $f$).
Note that $A_{max}$ is $G$-invariant and that in literature the set $A_ {max}$ is often referred to as \textit{tension graph}.

As we have seen in the discussion above, for every map there is an induced structure. More specifically, if $A,B \in \mathcal{O}$ and $f : A \rightarrow B$ is a PL-map, then\textbf{ the pre-train track structure induced by $f$} on $A$ is defined by declaring germs of edges to be equivalent if they have the same \textit{non-degenerate} $f$-image (so if two maps that are collapsed by $f$, they are not equivalent).

We are now in position to define optimal maps:
\begin{defin}
Let $A,B \in \mathcal{O}$. A PL-map $f : A \rightarrow B$ is not optimal at $v$, if $A_{max}$ has only one gate at $v$ for the pre-train track structure induced by $f$. Otherwise, $f$ is \textbf{optimal at $v$}. The map $f$ is \textbf{optimal}, if it is optimal at all vertices.
\end{defin}
\begin{remark}
A PL-map $f : A \rightarrow B$ is optimal if and only if the pre-train
track structure induced by $f$ is a train track structure on $A_{max}$. In particular, if $f : A \rightarrow B$ is an optimal map, then at every vertex $v$ of $A_{max}$ there is a legal turn in $A_{max}$.
\end{remark} 
Note also that by \cite{FM}, every  PL-map is optimal at non-free vertices and for every $A,B \in \mathcal{O}$ there exists an optimal map from $A$ to $B$. Therefore we can always choose our $\mathcal{O}$ - maps to be optimal and we will use optimal maps without further mention.

\subsection{Relative Automorphisms}

We denote by $Out(G, \{ G_i\} ^{t})$ the subgroup of $Out(G, \mathcal{O})$ made of those automorphisms that act as a conjugation by an element of $G$ on each $G_i$.
Since the $G_i$'s are free factors of $G$, each subgroup $G_i$ is equal to its normalizer in $G$. Therefore, any element of $Out(G, \mathcal{O})$ (i.e. that preserves the conjugacy class of the $G_i$'s)
induces a well-defined outer automorphism of $G_i$. Therefore there is a natural homomorphism
$Out(G, \{ G_i \} ^{t}) \rightarrow Out(G_i)$ and by taking the product over all groups $G_i$, we get a (surjective) homomorphism  $Out(G, \mathcal{O}) \rightarrow \bigoplus \limits ^{p} _{i = 1} Out(Gi)$, with kernel exactly $Out(G, \{G_i \} ^{t})$.

\subsection{Train Track Maps and Irreducibility}
In this section we will define the notion of a "good" representative of an outer automorphism. It is a generalisation of train track representatives of automorphisms of free groups, but as we have already mentioned we work in the trees instead of their quotients. For more details for this approach see \cite{FM, Syk}. As we have seen there are representatives of every outer automorphism (i.e. $\mathcal{O}$-maps from $A$ to $\phi(A)$), but sometimes we can find representatives with better properties. These maps, which are called \textit{train track maps}, are very useful and every irreducible automorphism has such a representative (we can choose it to be simplicial, as well).\\
For $T \in \mathcal{O}$ we say that a Lipschitz surjective map $f : T \rightarrow T$ \textbf{represents}
$\phi$ if for any $g \in G$ and $t \in T$ we have $f(gt) = \phi(g)(f(t))$. (In other words,
if it is an $\mathcal{O}$-map from $T$ to $\phi(T)$.) We give below the definition of a train track map representing an outer automorphism. We are interested for these maps because we can control their cancellation (it is not possible to avoid it).

\begin{defin}
If $T \in \mathcal{O}$ then a PL-map $f : T \rightarrow T$, which representing $\phi$, is a train track map if there is a train track structure on $T$ so that
\begin{enumerate}
 \item  f maps edges to legal paths (in particular, $f$ does not collapse edges)
 \item If $f(v)$ is a vertex, then $f$ maps inequivalent germs at $v$ to inequivalent germs at $f(v)$.
 \end{enumerate}
\end{defin}

In the free case, an automorphism $\phi$ is called  \textit{irreducible}, if it there is no $\phi$-invariant free factor up to conjugation (or equivalently the topological representatives of $\phi$ haven't non-trivial proper invariant subgraphs). In our case we know that the $G_i$'s are invariant free factors, but we don't want to have "more invariant free factors". More precisely, 
we will define the irreducibility of some automorphism \textit{relative} to the space $\mathcal{O}$ or to the free product decomposition.
\begin{defin}
We say $\Phi \in  Out(G, \mathcal{O})$ is $\mathcal{O}$-\textit{irreducible} (or simply irreducible) if for any $T \in \mathcal{O}$ and for any $f : T \rightarrow T$ representing $\Phi$, if $W \subseteq T$ is a proper $f$-invariant $G$-subgraph then $ G / W$ is a union of trees each of which contains at most one non-free vertex.
\end{defin}

We can also give an alternative algebraic definition, but we need the notion of a free factor system. Suppose that $G$ can be written as a free product, $G = G_1 \ast G_2 \ast ...G_p \ast G_{\infty}$. Then we say that the set $\mathcal{A} = \{ [G_i] : 1 \leq i  \leq p \}$ is a\textbf{ free factor system for $G$}, where $[A]$ = $\{ gAg^{-1} : g \in G \}$ is the set of conjugates of $A$.\\
Now we define an order on the set of free factor systems for $G$. More specifically, given two free factor systems $\mathcal{G} = \{ [G_i] : 1 \leq i  \leq p \}$ and $\mathcal{H} = \{ [H_j] : 1 \leq j  \leq m \}$, we write $\mathcal{G} \sqsubseteq \mathcal{H}$ if for each $i$ there exists a $j$ such that $G_i \leq gH_jg^{-1}$ for some $g \in G$. The inclusion is strict, and we write $\mathcal{G} \sqsubset \mathcal{H}$, if some $G_i$ is contained strictly in some conjugate of $H_j$. We can see  $ \{[G] \} $ as a free factor system and in fact, it is the maximal (under $\sqsubseteq$) free factor system. Any free factor system that is contained strictly to $\mathcal{G}$ is called \textbf{proper}. Note also that the Grushko decomposition induces a free factor system, which is actually the minimal free factor system (relative to $\sqsubseteq$). A more detailed discussion for the theory of free factor systems can be found in \cite{HL}.\\
We say that $\mathcal{G} = \{ [G_i] : 1 \leq i  \leq p \}$ is $\phi$ - \textbf{invariant} for some $\phi \in Out(G)$, if $\phi$ preserves the conjugacy classes of $G_i$'s. We are only interested for free factor systems that $G_{\infty}$ is a finitely generated free group. In particular, we suppose that $G = G_1 \ast G_2 \ast ... G_p \ast G_{\infty}$, and $G_{\infty}= F_k$ for some f.g. free group $F_k$. In each free factor system $\mathcal{G} = \{ [Gi] : 1 \leq i  \leq k \}$, we associate the outer space $\mathcal{O} = \mathcal{O}(G, (G_i)^{p}_{i=1}, F_k)$ and any $\phi \in Out(G)$ leaving $\mathcal{G}$ invariant, will act on $\mathcal{O}$ in the same way as we have described earlier.
\begin{defin}
Let $\mathcal{G}$ be a free factor system of $G$ which is $\Phi$- invariant for some $\Phi \in Out(G)$. Then $\Phi$ is called \textit{irreducible relative to $\mathcal{G}$}, if $\mathcal{G}$ is a maximal (under $\sqsubseteq$) proper, $\Phi$-invariant free factor system.
\end{defin}

The next lemma confirms that the two definitions of irreducibility are related.

\begin{lemma}
Suppose $\mathcal{G}$ is a free factor system of $G$ with associated space of trees $\mathcal{O}$, and further suppose that $\mathcal{G}$ is $\phi$-invariant. Then $\phi$ is irreducible relative to $\mathcal{G}$ if and only if $\phi$ is $\mathcal{O}$-irreducible.
\end{lemma}

Moreover, one interesting fact is that for an irreducible automorphism we can give a characterisation of train track maps using the axes of hyperbolic elements. More specifically, if $\phi$ is irreducible, then for a map $f$ representing $\phi \in Out(G, \mathcal{O})$, to be a train track map is equivalent to the condition that there is $g \in G$ (hyperbolic element) so that $L = axis_T (g)$ (the axis of $g$) is legal and $f^k(L)$ is legal $k \in \mathbb{N}$.
\\
\\
Now let's give the definition of an irreducible automorphism with irreducible powers relative to $\mathcal{O}$, which are the automorphisms that we will study.
\begin{defin}
An outer automorphism $\phi \in Out(G, \mathcal{O})$ is called \textbf{IWIP} (\textit{irreducible with irreducible powers} or \textit{fully irreducible}), if every $\phi ^{k}$ is irreducible relative to $\mathcal{O}$.
\end{defin}

The next theorem is very important since we can always choose representatives of irreducible automorphisms with nice properties, as in the free case. It generalises the well known theorem of Bestvina and Handel (see \cite{BH}) . In particular, we can apply it on every power of some $IWIP$.
\begin{theorem} [Francaviglia- Martino]
Let $\phi \in Out(G, \mathcal{O})$ be irreducible. Then there exists a (simplicial) train track map representing $\phi$.
\end{theorem}

The discussion above implies that we can always find an optimal train track representative of an irreducible $\phi \in Out(G, \mathcal{O})$. This map has the property that the image of every legal path (in particular, of edges) is stretched by a constant number $\lambda \geq 1$ which depends only on $\phi$.\newline

We close this subsection with an interesting remark.
\begin{remark}
Every outer automorphism $\phi \in Out(G)$ is irreducible relative to some appropriate space (or relative to some free product decomposition). 
\end{remark}

\subsection{Bounded Cancellation Lemma}

Let $T, T′ \in \mathcal{O}$ and $f : T \rightarrow T'$ be an $\mathcal{O}$- map. If we have a concatenation of legal paths $ab$ where the corresponding turn is illegal, then it is possible to have cancellation in $f(a)f(b)$. But the cancellation is bounded, with some bound that depends only on $f$ and not on $a,b$ . In particular,  we can define the bounded cancellation constant of $f$ (let's denote it $BCC(f)$) to be the supremum of all real numbers $N$ with the property that there exist $A, B, C$ some points of $T$ with $B$ in the (unique) reduced path between $A$ and $C$ such that $d_{T′}(f(B), [f(A), f(C)]) = N$ (the distance of $f(B)$ from the reduced path connecting $f(A)$ and $f(C)$ ), or equivalently is the lowest upper bound of the cancellation for a fixed $\mathcal{O}$-map.\\
The existence of such number is well known, for example a bound has given in \cite{H1}:
\begin{lemma}
Let $T \in \mathcal{O}$, let $T′ \in \mathcal{O}$, and let $f : T \rightarrow T'$ be a Lipschitz map. Then $BCC(f) \leq Lip(f) qvol(T)$, where $qvol(T)$ the quotient volume of $T$, defined as the infimal volume of a finite subtree of $T$ whose $G$-translates cover $T$.
\end{lemma}

We can also, exactly as in the free case, define a critical constant, $C_{crit}$ corresponding to a train track map.\\
Let's suppose that $f$ is train track map with expanding factor $\lambda$ (for example, a train track representative of some IWIP $\phi$). If we take $a,b,c$ legal paths and $abc$ is a path in the tree, and let's denote  $l =length(b)$ the length of the middle segment. If we suppose further that satisfies $\lambda l - 2BCC(f) > l$, then iteration and tightening of $abc$ 
will produce paths with the length of the legal leaf segment corresponding to $b$ to be arbitarily long. This is equivalent to require that $l > \frac{2BCC(f)}{\lambda -1}$ , and we call the number $C_{crit} =  \frac{2BCC(f)}{\lambda -1}$, \textit{the critical constant for $f$}.\\
For every $C$ that exceeds the critical constant there is $m > 0$ such that $b$, as above, has length at least $C$ then the length of the legal leaf segment of  $[f^{k} (abc)]$ corresponding to $b$  is at least $m \lambda^{k} length(b)$.\\
Therefore we can see that any path which contains a legal segment of length at least $C_{crit}$, has the property that the lengths of reduced $f$-iterates of the path are going to infinity.

\subsection{N-periodic paths}

A difference between the free and the general case is that it is not always the case that there are finitely many orbits of paths of a specific length (if there are non-free vertices with infinite stabiliser), but it is true that there are finitely many paths that have different projection in the quotient. Therefore the role of Nielsen periodic paths play the N-periodic paths that we define below. Note that if $h: S \rightarrow S$, we say that a point $x \in S$ is $h$-periodic, if there are $g \in G$ and some natural $k$  s.t. $h^{k}(x) = gx$.

\begin{defin}
\begin{enumerate}
\item Two paths $p, q$ in $S \in \mathcal{O}$ are called \textit{equivalent}, if they project to the same path in the quotient $G/S$. In particular, their endpoints $o(p), o(q)$ and $t(p), t(q)$ are in the same orbits, respectively.
\item Let $h : S \rightarrow S$ be a representative of some outer automorphism $\psi$, let $p$ be a path in $S$ and let's suppose that the endpoints of $p$ are $h$ - periodic (with period $k$), then we say that a path $p$ in $S$ is \textit{ N-periodic} (with period $k$), if the paths $[h^{k}(p)], p$ are equivalent.
\end{enumerate}
\end{defin}

\textbf{Geometric and non-Geometric automorphisms}:
We will define here some notions for automorphisms that have been motivated by the properties of geometric and non-geometric automorphisms, respectively. The terminology also comes from the free case. In that case, we say that $\phi$ is geometric if it can be represented as a (pseudo-Anosov) homeomorphism of a punctured surface. It is well known that for the non-geometric case there is an integer $m$ such that it is impossible to concatenate more than $m$ indivisible Nielsen paths for every map $f$ which represents $\phi$. We will generalise this property in order to give our definitions, using the notion of an indivisible N-periodic path as in the free case. In particular:

\begin{defin}
We say that some $\phi$ has the $NGC$ property , if it is impossible to concatenate more than $m$ indivisible N-periodic paths for every $\mathcal{O}$-map $f$ which represents $\phi$. Otherwise, we say that $\phi$ has the $GC$ property.
  \end{defin}

\subsection{Relative train-track maps}
Having good representatives of outer automorphisms, is very useful. If our automorphism is irreducible, it is possible to find train track representatives, as we have seen. But even in the reducible case we can find relative train track representative. The existence of such maps it follows from \cite{FM} or \cite{CT}.\\
That we have is that every automorphism can be represented as an $\mathcal{O}$-map $f : T \rightarrow T$ such that $T$ has a filtration $T_0 \subseteq T_1 \subseteq ...\subseteq T_k = T$ by $f$-invariant $G$-subgraphs, where $T_0$ contains every non- free vertex, we denote  by $H_r = cl(T_r - T_{r-1})$ and we suppose that the transition matrix (it can be defined as in the free case but we count orbits of edges) of every $H_r$ is irreducible (or zero matrix) so we can correspond in every $H_{r}$ some PF eigenvalue (let's denote it $\lambda _r$ ) . In addition, $f$ has some train track properties (such as mixed turns are legal and the map is $r$-legal). There is a very interesting corollary that we will use: for every edge-path $a$ in $H_r$, the reduced image of $a$, $[f(a)]$, can be written as a concatenation of non-degenerate edge-paths in $T_{i-1}$ and $H_i$ with the first and the last contained in $H_i$.\\
For such $a$, we can distinguish between two cases for the strata: if there exists some edge of $e$ in $H_r$ such that $[f(e)]$ contains at least two copies (orbits) of $e$, then we say that the stratum is \textit{exponentially growing} and we can see the $r$-lengths of images of edges in $H_r$ expands by $\lambda _r >1$ and in particular the lengths of reduced $f$-iterates of edges in $H_r$ are going to infinity (using the train track properties). Otherwise, the stratum called non-exponentially growing and the map $f$ (if we ignore the lower strata) is just a permutation of edges of the same length. An automorphism is called \textit{exponentially growing} if some representative has at least one exponentially growing stratum. In other case, it is called \textit{non-exponentially growing} automorphism. 

\subsection{Graph of Groups and Subgroups}
We will recall only some facts for the graph of groups. For more about graph of groups and their subgroups, see \cite{S}.\\
In the special case that we are interested, a graph of groups can be defined as a finite connected graph $X$ (let call $\Gamma$ the underlying graph) for which in every vertex $v$ we correspond some (vertex) group $G_v$. We call \textit{non -free} the vertices for which the corresponding group is non-trivial. Then the fundamental group of $X$, $\pi _1 (X)$ is the free product of $\pi _1 (\Gamma)$ (which is a f.g. free group) and the vertex groups.\\
We will use a specific kind of subgroups of $\pi _1 (X)$. Let $\gamma$ be a loop in $v_0 \in V(\Gamma$). Then starting from $v_o$ and following the path of $\gamma$ we meet some non-free vertices (we can return back also, but we have always follow $\gamma$). So we can read words of a fixed form, and this process produces words of the fundamental group (we can see it as the group which it consists of all the words constructed as above but without fixing some loop $\gamma$). In fact, the set of all such words is a subgroup of $\pi_1(X)$, which corresponds to $\gamma$.

\section{Every two $\mathcal{O}$ - maps coincide}
In \cite{FM} it has been proved the existence of $\mathcal{O}$-maps. We will prove that even if in the construction of such maps there is a lot of freedom, the reduced images of all of them coincide, up to bounded error. As a consequence we obtain that their lengths are comparable.

\begin{theorem}\label{O-maps}
	Let $f,h : A \rightarrow B$ be $\mathcal{O}$ - maps. Then there exists a positive constant $C$ (which depends only on $f$, $h$ and $A$), so that for every path $L$ in $A$, then  $[f(L)]$ and $[h(L)]$ are equal, except possibly some subpaths near their endpoints which their lengths are bounded by $C$.
\end{theorem}

\begin{proof}
	Firstly, we suppose that there is at least one non-free vertex which we denote it by $v$. Then we have that $f(v) = h(v)$. If $L = [a,b]$ is an edge - path,
	then in distance at most $vol(A/G)$, we can find vertices of the form $g_1 v,g_2 v $ near $a, b$ respectively such that
	$[a,b] \subseteq [g_1v, g_2 v]$. Then $[f(L)]$ is contained in $[f(g_1 v), f(g_2 v) ]$, except possibly some segments near $a, b$ of length at most  $C' = vol(A/G) Lip(f)$. Similarly, we apply the same argument for $[h(g_1 v), h(g_2 v)]$ and we get a constant $C'' = vol(A/G) Lip(h)$.
	Therefore since $[h(g_1 v), h(g_2 v)] = [f(g_1 v), f(g_2 v) ]$, we get $[f(L)] = [h(L)]$ except possibly some segments near $a,b$ which are bounded by $C = max(C', C'')$ (by definition depends only on $Lip(f), Lip(h),
	 vol(G/A)$)\\
	If there are no non- free vertices, we are in the free case and the result is well known.
\end{proof}

Note also that it is not difficult to see that every $\mathcal{O}$-map is a quasi-isometry.

\section{Laminations}
We follow exactly the same approach as in \cite{BFM} and some of the proofs are essentially the same, but since in this context the definitions have adjusted appropriately, we give detailed proofs for the convenience of the reader. On the other hand, there are a lot of technical issues which are not appeared in the free case and they are addressed separately.\\
In this section we define the notion of the lamination associated to an IWIP. Firstly, we use the train track maps to define the lamination in a specific tree and the existence of $\mathcal{O}$-maps between any two trees allows us to generalise it for every tree.
\subsection{Construction of the lamination and properties}
Let $ \phi \in Out( G , \mathcal{O}) $ be an (expanding) irreducible automorphism, with irreducible powers and $ f : A \rightarrow A$ for some $A \in \mathcal{O}$ be a train track map which represents $\phi$ (so it satisfies $f(gx) = \phi(g) f(x)$).  We can also suppose that $f$ expand the length of the edges by a uniform factor $ \lambda > 1$ (this can be done if we choose an optimal train track that represents $f$, as we have already seen).\\
By changing $f$ with some iterate, if necessary, we can suppose that there is $x \in A$ which is a periodic point ($f^{k}(x) = x$, for some k), in the interior of some edge
(in general there exists $x$ s.t. $f^{k}(x) = gx$ since the quotient is finite, but we can change the space $A$, changing isometrically the action, with $\phi_{g} (A)$ and there the requested property holds).
 Now let $U$ some $\epsilon$-neighbourhood, for some small $\epsilon$ (we want the neighbourhood to be contained in the interior of the edge) and then there is some $N > 0$ s.t. $f^{N} (U) \supset U$.\\
We can choose an isometry $ \ell : (- \epsilon, \epsilon) \rightarrow U$ and extend it to the unique isometry $\ell : \mathbb{R} \rightarrow A$ s.t. $\ell( \lambda ^{N} t  ) = f^{N} (\ell(t))$ and then we say that the bi-infinite line $\ell$ is \textit{obtained by iterating a neighbourhood} of $x$.\\

\begin{defin}
\begin{itemize}
\item We say that two isometric immersions $A : [a,b] \rightarrow A$ and $B :[c,d] \rightarrow A$, where $a,b,c,d \in \mathcal{R}$ are equivalent, if there exists an isometry $q : [a,b] \rightarrow [c,d]$ s.t the triangle  commutes ($Bq = A$). (This relation is an equivalence relation on the set of isometric immersions from a finite interval to $A$).
 \item If $P$ is an equivalence class and we choose a representative of that class $\gamma : [a,b] \rightarrow A$, we can define $f(P)$ as the equivalence class of $f \gamma : [a, b] \rightarrow A$, pulled tight and scaled so it is an isometric immersion.
 \item A leaf segment of an isometric immersion $\mathbb{R} \rightarrow A$ is the equivalence class of the restriction to a finite interval.
\end{itemize}
\end{defin}

Let $\ell$ be an isometric immersion, then we correspond the $G$-set $I_{\ell}$ (of the leaf segments of $\ell$) to $\ell$. We can also define an equivalence relation on the set of isometric immersions from $\mathbb{R}$  to $A$.\\

\begin{defin}
Let $\ell, \ell '$ be two isometric immersions from $\mathbb{R}$  to $A$, then we say that they are equivalent if $I_{\ell} = G I_{\ell '}$. Namely, we say that they are \textit{equivalent} if for every leaf segment $P$ of $\ell$ there is an element $g \in G$ and $Q$ a leaf segment of $\ell '$ s.t. $P = gQ $ and vice versa (or equivalently every l.s. of $\ell$ is mapped by some $g$ to a l.s. of $\ell '$)
\end{defin}

\begin{remark}
 Here note that it is obvious that if $\ell(t) = g \ell '(t) $ ($\ell, \ell '$ are in the same orbit), then $\ell$ and $\ell '$ are equivalent.\end{remark}

We will prove that if we construct any other line by iterating a neighbourhood of any other periodic point (here we mean that there is $k$ and $g \in G$ s.t. $f^{k}(x) = gx$) then it is equivalent with $\ell$.
\begin{lemma}
Let $y \in A$, be any other $f$-periodic point in the interior of some edge of $A$ and $\ell '$ is the obtained by iterating of some neighborhood of $y$.
Then $\ell$ and $\ell '$ are equivalent.
\end{lemma}

\begin{proof}
We will show that any l.s. of $\ell$ is mapped by some element of $G$ to a l.s. of $\ell '$, then the converse follows by symmetry.\\
Since $f$ represents an irreducible automorphism (and the same holds for every power of $f$), $\ell '$ contains some orbit of every edge, so in particular if $x$ is contained in the interior of the edge $e$ we have that there exists some $g \in G$, s.t. $gx \in ge \subseteq \ell '$. So there is an isometry $\psi : (-\epsilon, \epsilon) \rightarrow ( a -\epsilon, a + \epsilon)$ with the property $\ell(t) = g \ell '(\psi (t) )$.\\
Let $N'$ be a natural number s.t. $\ell '( \lambda^{N'} t  ) = f^{N'} (\ell(t))$ and then for any $t \in U$ ($U$ as in the definition) we have that $\ell( \lambda ^{kNN'} t  ) = f^{kNN'} (\ell(t)) = f^{kNN'} (g \ell ' (\psi(t))) = \phi^{kNN'}(g)f^{kNN'} (\ell ' (\psi(t)))  = \phi^{kNN'}(g) \ell'(\lambda^{kNN'} \psi(t))$.\\
But since every prechosen interval is contained in some interval of the form $\lambda ^{kNN'} (-\epsilon, \epsilon)$ for large $k$, we have that for every l.s. of $\ell$ is mapped by some $\phi^{kNN'} (g) \in G$ to some l.s. of $\ell '$.
\end{proof}

We are now in position to define the stable lamination corresponding to $A$. \\

Now the \textbf{stable lamination} in $A$-coordinates $\Lambda = \Lambda_{f}^{+} (A)$ is the equivalence class of isometric immersions from $\mathbb{R}$ to $A$ containing some (and by previous lemma any) immersion obtained as above (by iterating a neighborhood of a periodic point). We call the immersions
representing $\Lambda$ \textbf{leaves} of $\Lambda$ and the leaf segments (l.s.) of some leaf of $\Lambda$ \textbf{leaf segments} of $\Lambda$ (by definition of the equivalence relation, every leaf of $\Lambda$ contains some orbit of every l.s. of $\Lambda$).
\\
Note that the every leaf of the lamination project to the same bi-infinite path in the quotient.\\
We will list some useful properties of the stable lamination.

\begin{prop}
\begin{enumerate}
  \item Any edge of $A$ is a leaf segment of $\Lambda$.
  \item Any $f$-iterate of a leaf segment is a leaf segment.
  \item Any subsegment of a leaf segment is a leaf segment.
  \item Any leaf segment is a subsegment of a sufficiently high iterate of an edge.
  \item For any leaf segment $P$ there is a leaf segment $P'$ such that $f(P') = P$.
  \item Let $a$ be a segment which is the period of the axis of some hyperbolic element which crosses $k$ edges (counted with multiplicity). Then any $f$-iterate of $a$ (pulled tight) can be written as concatenation of less or equal $k$ leaf segments.
\end{enumerate}
\end{prop}
\begin{proof}

\begin{enumerate}
  \item This is clear by the proof of the previous lemma, since $f$ represents an irreducible automorphism and this implies that every $\ell$ contains orbits of every edge, so if $ge$ is contained in $\ell$ then $e$ is contained in $g^{-1} \ell $ which is equivalent to $\ell$ thus is a leaf of $\Lambda$, and as consequence $e$ is leaf segment of a leaf therefore it is l.s. of $\Lambda$.
  \item Firstly, we note that if $x$ is $f$-periodic then $f(x)$ is $f$- periodic with the same period(in fact every $f^{m}(x)$ is periodic) and let's denote $\ell '$ the isometric immersion constructed as above, so if $P$ is a l.s. of $\ell$, then $f(P)$ is a l.s. of $\ell '$ but since $\ell,\ell '$ are equivalent by lemma, we have that $\ell '$ is a leaf of $\Lambda$ and therefore $f(P)$ is a l.s. of $\Lambda$. So we can do it for every iterate of $f$.
  \item This is obvious, since we restrict the isometric immersion to the subsegment and it is a l.s. of a leaf of $\Lambda$ and as a consequence a l.s. of $\Lambda$.
  \item We have that $f$ expands the length of every edge by $\lambda$, but we can use for representative the isometric immersion constructed as above (by iterating a periodic neighborhood) and the edge in which the periodic point belongs, then by construction of $\ell$ every l.s. is contained in an high iterate of this edge. For any other representative $\ell '$ now we can translate $\ell$ as above (by some element $g \in G$) to have a common segment that contain the prechosen l.s. and the proof reduced to the first case.
  \item Let $P$ be a l.s. of $\Lambda$. By (iv) we have that there exists some iterate of an edge and so by $\ell$ an iterate of a l.s. $P''$ s.t. $P$ is contained in $f^{m}(P'')$ and since iterates of l.s. are l.s. and subsegments are l.s. as well, we have that there is $P'$ subsegment of $f^{m-1} (P'')$ with the property $P = f(P') $
  \item This is obvious since edges are l.s. and $f$-iterates of l.s. are l.s..
\end{enumerate}
\end{proof}

We note that (ii) implies that $f^k (\ell)$ is a leaf of the lamination, for every $k$.

\begin{defin}
We say that a sequence ${a_{i}}$ of isometric immersions $[0,1]_{i} \rightarrow A$ (where the metric on $[0,1]_{i}$ is scalar multiple of the standard part which depends on i), (weakly) converges to $\Lambda$,\\
if for every $L > 0$ the ratio,
\begin{equation*}
\frac{m(\{x \in [0,1]_{i} | \text{the L- nbhd of x is a leaf segment} \})}{m([0,1]_{i})}
\end{equation*}
converges to 1.
\end{defin}

\begin{prop}\label{per1}
Suppose that $a$ is a segment in $A$ which corresponds to the period of the axis of some hyperbolic element, which is not N-periodic. Then the sequence (of tightenings of $f^{i}(a)$), $[f^{i} (a)]$ weakly converges to $\Lambda$.
\end{prop}

Note that such hyperbolic elements always exist. For example the basis elements of the free group, are not  N-periodic by definition of irreducibility. 

\begin{proof}
Suppose that $a$ can be written as a concatenation of $k$ l.s. then we have $k-1$ illegal turns (we don't count the endpoints) and since $f$ is a train track map we have that the number of illegal turns in $[f^k (a)]$ is non-increasing so it contains less than or equal to $k-1$  l.s.. Therefore if the lengths of reduced iterates of $a$ is bounded, and since there are finitely many inequivalent paths with length less than or equal to a specific number, we have that $a$ is N-preperiodic and therefore periodic because $a$ corresponds to a group element, which leads to a contradiction to the hypothesis. Therefore some $[f^{i}(a)]$ contains arbitarily long legal segments ($> C_{crit}$), and since the length of $[f^{j}(a)]$ expands for large $j$,  we have that there are finitely many $L$-nbds contain points without the requested property (of the endpoints of the concatenation of l.s. so at most $k$) and the measure of these is at most $2Lk$, as a consequence the ratio converges to $1$.
\end{proof}

\begin{defin}
An isometric immersion $l : \mathbb{R} \rightarrow A$ is quasiperiodic  (qp), if for every $L > 0$ there exists $L' > 0$ s.t. for every l.s. $P$ of $\ell$ of length $L$ and for every l.s. $Q$ of length $L'$ there is $g \in G$ s.t. $gP \subseteq Q$ ($P$ is mapped by $g$ to a subsegment of $Q$).
\end{defin}

\begin{prop}
Every leaf of $\Lambda$ is quasiperiodic.
\end{prop}
\begin{proof}
We will first prove it for some $\ell$ which has constructed by iterating neighbourhood of a periodic point.\\
We first verify it for leaf segments $\Pi$ that consists of only two edges. \\
If we choose $L_0 > 2 max_{e}(len(e))$, then if a l.s. $P$ has length $\geq  L_0$, then it contains a subleaf segment which is an edge. Then there is $N$(we can also choose it to be multiple of $k$) s.t. $f^{N}$ restricted to any edge crosses some orbit of every turn that they crossed by leaves of $\Lambda_{f}^{+}(A)$. So in particular for the chosen $\Pi$ the iterate of $f$ takes the orbit of that turn, so there exists $g \in G$ such that $\Pi \subseteq g f^{N}(P)$.\\
Now if $P'$ is any l.s. of length $\lambda^{N} L_{0}$, then $P' = f^{N}(P)$ for some $P$ l.s. of length $L_{0}$ and therefore $\Pi \subseteq g P'$.\par
For the general case, let $L > 0$ be given, then there is $M > 0$ (we choose it to have the property $\lambda^{-M}L < 2 min(len(e))$) s.t. any l.s. of length $\leq \lambda^{-M}L$ is a subsegment of a two-edge l.s. as above and let $L' = \lambda^{M+N} L_0$.\\
So let $P$ be a l.s. of length $L$ and $P'$ be a l.s. of length $L'$. Then by the properties we have that $P = f^{M}(\Pi)$ where $\Pi$ is contained to a l.s. as in the special case(by the choice of $M$, since $\Pi$ has length $\lambda^{-M}L$), and similarly $P' = f^{M}(\Pi ')$ for a l.s. $\Pi '$ of length $\lambda^{N}L_{0}$. By the special case we have that $\Pi \subseteq g \Pi '$ and this implies that $P = f^{M}(\Pi) \subseteq \Phi^{M}(g) f^{M}(\Pi') = \Phi^{M}(g)P'$. Since $\ell$ is $\Phi^{M}(g)$-invariant, we have the requested property.\\
For any other equivalent isometric immersion $\ell '$, if we have $P$ l.s. of length $L$ and $Q$ l.s. of length $L'$ then we can find an isometric immersion $\ell$ like the first case with $Q$ as common segment. Then by the equivalence there exists $g_{1} \in G$ s.t. $g_{1}P$ is l.s. of $\ell$, and by quasiperiodicity of $\ell$, there is $g_{2}$ s.t. $g_{2}P \subseteq Q$ and $g_2 P$ is a l.s. of $\ell '$, so we have that $\ell '$ is quasiperiodic.
\end{proof}

\subsection{Lamination in every tree}

Suppose that $f : A \rightarrow A$ and $\Lambda_{f}^{+}(A)$ as above and $B \in \mathcal{O}$. Then we know that there exists an optimal map(in particular $\mathcal{O}$ -map) $\tau : A \rightarrow B$. Then for any immersion $\ell : \mathbb{R} \rightarrow A$ we denote by $\tau (\ell) : \mathbb{R} \rightarrow B$ the unique (up to precomposition by an isometry of $\mathbb{R}$) pulled tight to be the isometric immersion corresponding to $\tau \ell$.
\begin{lemma}
\begin{itemize}
  \item If $\ell, \ell ' : \mathbb{R} \rightarrow A $ are equivalent leaves, then $\tau (\ell), \tau (\ell ')$ are equivalent.
  \item If $\ell$ is quasiperiodic, then $\tau (\ell)$ is quasiperiodic.
\end{itemize}
\end{lemma}

\begin{proof}
Every optimal map $\tau$ by \cite{FM}, can be factored as the composition of a homeomorphism and a finite sequence of folds. We have just to prove that the lemma is true for homeomorphism and folds.\\
Firstly, let suppose that $\tau$ is homeomorphism. In particular $[\tau(\ell)] = \tau (\ell)$ and the same holds for $\ell '$ as well.\\
Let $P'$ is a l.s. of $\tau (\ell)$, then there is some l.s. of $\ell$ $P$ s.t. $P' = \tau (P)$, so there is a translation of $P$ by some element of the group, $gP$ which is contained in $\ell '$, therefore $\tau(gP) = g P' $ is contained in $\tau (\ell ')$. The converse follows by symmetry and so $\tau (\ell)$ and $\tau (\ell ')$ are equivalent.\\
Suppose now that $\ell$ is quasiperiodic, fix a $L > 0$ let $P'$ l.s. of $\tau (\ell)$ of length $L$. Then there is a l.s. $P$ of length at most $ K$ (by Bounded Cancellation Lemma there exists such $K$ which doesn't depend on $P$ but only on $L$) s.t. $P' = \tau (P)$. Then we can define $L'' = L' Lip( \tau)$, where $L'$ is the constant corresponding by quasiperiodicity to $K$ and we have that if we choose any $Q'$ l.s. of $\tau (\ell)$ of length $L''$ then there exists a l.s. $Q$ of $\ell$ of length at least $L'$ s.t. $\tau (Q) = Q'$. Then $Q$ contains orbits of any l.s. of length at most $K$, in particular it contains some translation of $P$ for some $g \in G$ and therefore as above $Q'$ contains some translation  of $P'$. So $\tau (\ell)$ is quasiperiodic.\\
We suppose that $\tau$ is an equivariant isometric simple fold of some segments starting from the same point $v$ and has the same $\tau$ - image, let call them $a,b$ and $c$ be the corresponding segment in the quotient.\\
For the first statement, we note that is obvious for a l.s. of $[\tau(\ell)]$ which don't contain some orbit of $c$, since there $\tau$ is the identity. On the other hand, if $P'$ is l.s. of $\tau (\ell)$ which contains some orbits of $c$, then there exists $P$ which contain the same number of orbits as the folded turn and $[\tau(P)] = P'$ (it is concatenation of the segments before and after the folds). Since $\ell, \ell '$ are equivalent we have that we can find $g \in G$ s.t. $gP$ is contained in $\ell '$, then $[\tau(gP)]$ is a a l.s. of $[\tau(\ell ')]$. But $[\tau(gP)]$ is just a translation (by $g$) of $\tau(P)$, and therefore as above we obtain that $[\tau(\ell)], [\tau(\ell ')]$ are equivalent.\\
For the quasiperiodity of $[\tau(\ell)]$ we fix a number $L>0$ and we call $M$ the maximum number of orbits of $v$ which there are in a segment of length $L$, and $L'$ is the number corresponds by quasiperiodicity for  $L'' = L + 2 M len(a)$. Now let $P'$ be a l.s. of length $L$, then there is $P$ which contains the same number of orbits of the folded turn and $[\tau(P)] = P'$ as above. Then $P$ has length at most $L''$, some translation of it is contained in every l.s. of $\ell$ of length $L'$. Now let choose $Q$ any l.s. of $[\tau (\ell)]$ of length $L'$ then the preimage has length at least $L'$, and therefore the preimage has the requested property. So $Q$ contains a translation of $P'$ as above.
\end{proof}

\begin{defin}
The stable lamination of $f : B \rightarrow B$ in the $B$-coordinates is the equivalence class $\Lambda_{f}^{+}(B)$ containing $\tau (\ell)$ for some (and by previous lemma any) leaf of $\Lambda _{f}^{+}(A)$.
\end{defin}
Using again the property that $\tau$ is factored as the composition of a homeomorphism and a finite sequence of folds combined with the result for the
$\Lambda_{f}^{+}(A)$, we have the following proposition.

\begin{prop}\label{per2}
Let $a$ be a segment which is the period of the axis of a hyperbolic element in $A$, which is not N- periodic. Then the sequence $\{[\tau(f^{i}(a)) ] \}$ weakly converges to $\Lambda_{f}^{+}(B)$
\end{prop}

\begin{lemma}
Suppose that $h : B \rightarrow B$ is any other train track map representing $\Phi$. Then $\Lambda _{f}^{+}(B) = \Lambda _{h}^{+}(B)$
\end{lemma}

\begin{proof}
Let $a$ be a periodic segment as in \ref{per1} and \ref{per2}. Then the sequences $[\tau(f^{i}(a))], [h^{i}(\tau(a))]$ weakly converge to $\Lambda_{h}^{+}(B)$ and to $\Lambda_{f}^{+}(B)$, respectively by the previous propositions. But $\tau f^{i}, h^{i} \tau$ are $\mathcal{O}$-maps from $A$ to $\phi(B)$, so their reduced images coincide in every path,  after deleting some bounded segments near endpoints. Then there are leaves $\ell, \ell '$ of $\Lambda_{h}^{+}(B)$ and $\Lambda_{f}^{+}(B)$ respectively with arbitrarily long common leaf segments. Since they are both quasiperiodic, it follows that they are equivalent. Indeed, let $P$ be a l.s. of $\ell$ of length $L$ then there exists $L'$ s.t. for every l.s. of length $L'$, $P'$ there is some $g \in G$ s.t. $P \subseteq gP'$. But we can find a common segment $Q$ of $\ell$ and $\ell '$ of length at least $L'$, so by quasiperiodicity $P \subseteq gQ \subseteq \ell$ and since $Q \subseteq \ell '$ we have that $P \subseteq gQ \subseteq g \ell ' $ and therefore $g^{-1}P \subseteq \ell '$.\\
We have proved that for every l.s. of $\ell$ there is an element of the group that map this l.s. to a l.s. of $\ell '$ and similarly we can prove the converse so $\ell$ and $\ell '$ are equivalent by definition. Therefore $\Lambda_{h}^{+}(B) = \Lambda_{f}^{+}(B)$
\end{proof}

So we have proved that we can use any train track representative to define the set of laminations, in particular we give the following definition:
\begin{defin}
The \textbf{stable lamination} $\Lambda_{\Phi}^{+}$ associated to some IWIP $\Phi \in Out(G, \mathcal{O})$ is the collection $\{ \Lambda_{f}^{+}(B) | B \in \mathcal{O} \}$ where $f : A \rightarrow A$ is a train track representative of $\Phi$. The \textbf{unstable lamination} $\Lambda_{\Phi}^{-}$ of $\Phi$ is the stable lamination of $\Phi ^{-1}$.
\end{defin}

\section{Action}
Let $\phi$ be an IWIP and $f: T \rightarrow T$ be an optimal train track representative of $\phi$.\\
We denote by $\mathcal{IL}$ the set of stable laminations $\Lambda_{\phi}^{+}$, as $\phi$ ranges over all IWIP automophisms relative to $\mathcal{O}$.
The group $Out(G, \mathcal{O})$ acts on $\mathcal{IL}$ via
\begin{equation}
\psi \Lambda_{\phi}^{+} = \Lambda_{\psi \phi \psi ^{-1}}^{+}
\end{equation}
More specifically, if $\ell$ is a leaf of $\Lambda_{\phi}^{+}$ in the $S$-coordinates and $h : S \rightarrow S$ an $\mathcal{O}$ map representing $\psi$, then $[h(\ell)]$ represents a leaf of $\Lambda_{\psi \phi \psi ^{-1}}^{+}$.\\
We are interested to study the stabiliser of the action for a fixed automorphism. Note that obviously the centraliser, which we denote by $C(\phi)$, of the IWIP $\phi$ in $Out(G, \mathcal{O})$ is a subgroup of $Stab(\Lambda)$.\\
We will equip $T$ with a specific train-track structure, the\textit{ minimal train-track structure}; more specifically we declare a turn legal, if it is crossed by some leaf of $\Lambda_{f}^{+}$. The properties of the lamination imply that a turn is legal iff there is a $f$-iterate of an edge of $T$ that crosses the turn.\\

\section{Subgroups carrying the lamination}
From now and for the rest of the sections, we fix a group  of finite Kurosh rank $G$ with some (non-trivial) free product decomposition, the relative outer space $\mathcal{O}$ which corresponds to this decomposition, some expanding IWIP $\phi$ relative to $\mathcal{O}$ and the associated lamination $\Lambda^{+}_{\phi} = \Lambda $.\\
This section is devoted to prove that it is not possible for a proper subtree to contain every leaf of the lamination.  Moreover, we will prove that every relative train track representative of some automorphism of the stabiliser, after passing to some power, induces the identity on the quotient restricted to any proper invariant subraph (which is union of strata).

\begin{defin}
Let $A$ be a subgroup of $G$ of finite Kurosh rank, and let's denote $T \in \mathcal{O}$ and $T_A$ the minimal invariant $A$- subtree. We suppose also that for every $v \in  V(T_A)$, $Stab_A (v) = Stab_G (v)$. Then we say that $A$ \textit{carries the lamination $\Lambda$}, if there exist some leaf $\ell$ of $\Lambda$ which is contained in $T_A$.
\end{defin}

\begin{remark}
\begin{enumerate}
\item Every two leaves of the lamination project to the same bi-infinite path in $\Gamma$.
\item For every vertex $v$ of $T$ there exist $g \in G$ s.t. $gv \in T_A$ (in particular, $T_A$ contains some orbit of any non-free vertex).
\end{enumerate}
\end{remark}

\begin{prop}\label{carries}
If a $A$ is a subgroup of $G$, as in the previous definition, which carries $\Lambda^{+}_{\phi}$ then $A$ has finite index in $G$.
\begin{proof}
Let $f : T \rightarrow T$ be a train-track representative of $\phi$ , $\Gamma = G/T$ and let $H \rightarrow \Gamma$ be an isometric immersion corresponding to $A \leq G$. Then by our assumptions $H$ is finite graph of groups and by the remarks contains every non-free vertex. Therefore (using also the assumption that the corresponding vertex groups are full), we can complete the immersion, by adding vertices and edges, to a connected finite-sheeted covering space $p : \Gamma ' \rightarrow \Gamma$ and therefore we have that $T' = T$ (where $T'$ is Bass-Serre tree of $\Gamma ' $).\\
Now we know that if $A$ has infinite index, then we are really adding new edges in $\Gamma '$ or equivalently we add new orbits of edges in $T$. But then using irreducibility we can reach a contradiction.\\
More specifically, we choose $e$ (edge of $T$) such that $f(e)$ starts with $e$. Then for every $n$ the path $f^{n}(e)$ is a path of $T_A$. So if we choose any edge $e _1$ (lift of some edge in $\Gamma ' - H$) there does not exist $n$ and $g \in G$ such that $f ^{n} (ge ′)$ passes through $e_1$ (since $e_1$ is in different orbit of edges in $T_A$), but this contradicts the fact that the transition matrix corresponding to $f$, denote it by $A(f)$, is irreducible. As a consequence, $A$ must have finite index in $G$.
\end{proof}

\end{prop}

\begin{prop}\label{invariant subgraph}
Let $ \psi \in  Stab(\Lambda)$, and let $h : S  \rightarrow S$ be a relative train-track representative of $\psi$. Then we can find some $S' \in \mathcal{O}$ (which is topologically the same with $S$, but possibly with different marking) and a relative train track representative $h: S' \rightarrow S'$ of $\psi$ with the following property: let denote by $S_0$ some $h$-invariant $G$-subgraph of $S'$ (without free vertices of valence $1$) that is a union of strata. Then there is a $k$ s.t. if we restrict $h^k$ to $S_0$ induces the identity in $G / S_0$.
\end{prop}
\begin{proof}
	Let $\ell$  be a leaf in  $S$-coordinates and let $S_0$ be a proper $h$-invariant subgraph. The quasiperiodicity implies that there is an upper bound to the length of both $S_0$ and $S - S_0$ segments, and hence only finitely many segments occur (since there are finitely many lengths corresponding to edge- paths of bounded length in the quotient and quasiperiodicity implies that there are finitely many orbits of leaf segments of a specific length). Using the same argument we have that it is not possible for $\ell$ to contain arbitarily long segments of a proper subgraph since then the quasiperiodicity implies that $\ell$ is contained in that subgraph which contradicts to the previous proposition. Therefore $\ell$ is a concatenation of non-degenerate segments in $S_0$ and in $S - S_0$ (otherwise 
	would lift to a proper subgraph of $H$, which is impossible as we have already noticed). Now we have that all $S_0$-segments are
	$h$- preperiodic (there exist $M,N$ s.t. $h^{M}(L), h^{N}(L) $ are in the same orbit) or else $h$-iteration will produce arbitrarily long leaf segments contained in $S_0$ contradicting quasiperiodicity.\\
	We can start with the disjoint union $X$ of copies of the segments and the natural immersion $X \rightarrow S$ and we identify two endpoints of $X$ if they are mapped to the same point of $S$.
	Then fold to convert the resulting map to an immersion $ \pi : X' \rightarrow S$. But $\ell$ lifts to $X'$ (by construction) and so by previous proposition we have again that $X ' = S$ (it corresponds to a finite covering space of graph of groups). In particular, any simple periodic segment (the period of the axis of a simple loop or a loop corresponding to an element of some $G_{v_i}$ where $v_i$ has valence 1) in $S_0$ lifts to $X'$. Consequently, this segment is a concatenation of paths in $S_0$ each of which is $h$-preperiodic, and therefore this segment is N-periodic (since it corresponds to an element of the group and so we have inverse). Thus every such segment $a$ in $S_0$ is equivalent to some power $h^{k}(a)$ (note that there is a uniform bound for the powers) and hence for some $k$, $h^{k}$ restricted to $S_0$ induces the identity on the quotient, since $h$ is relative train track.

\end{proof}

\section{Stretching map}
In this section we will see that we can define a homomorphism from the stabiliser of the lamination to $\mathbb{R}$.

\begin{lemma}\label{lambda}
Suppose that $h: S \rightarrow S$ is an $\mathcal{O}$- map that represents $\psi \in Out(G, \mathcal{O})$. Then there exists a positive number $\lambda = \lambda(h, \Lambda)$ such that for every $\epsilon > 0$ there is $N > 0$ so that if $L$ is a leaf segment of $\Lambda$ of length $> N$, then $\big\vert \frac{\text{length}([h(L)])}{\text{length(L)}} - \lambda \big\vert < \epsilon$
\end{lemma}
\begin{proof}
We note that since $f$ is IWIP, we have that the transition matrix $M = A(f)$ is irreducible (as it is every power of $M$) and therefore we can apply the Perron - Frobenius theorem to $M$, as a consequence we have that long leaf segments of $\Lambda$ cross orbits of edges of $T$ with frequencies close to those determined by the components of the PF eigenvector.\\
Now fix large $k$ and then large l.s. are concatenation of l.s. of the form $f^{k}(e)$, for some edges of $T$, each orbit of edges with definite frequency.(For $k = 1$ this is the statement above, for $k>1$ apply $P.F$ theorem for $f^k$).\\
If $M$ is large enough, then for any l.s. $L$ with $length(L) > M$ we can think $L$ as concatenation of l.s. of the form $f^k (e)$ (there are possible some shorts segments contained in the first and the final segment, which are not of this form but we can ignore them since their contribution in lengths is neligible).\\
Now let $C$ be the bounded cancellation constant for $h : T \rightarrow T$, and let's denote $l_e = len(f^k (e))$, $l^{h} _{e} = len([h(f^k (e))])$, $N_e$ be the number of occurrences of orbits of $f^{k}(e)$ in $L$ and $N = \sum N_e$, then we have that  $\frac{N_e}{N} \rightarrow r_e$, as $len(L) \rightarrow \infty$ ($r_e$ is the PF component of the eigenvector that corresponds to $e$) by the PF theorem.\\
Note that the numbers $N_e, l_e, l_e ^{h}$ depends on $k$, so we define $a_k = \frac{\sum r_e l_e ^{h}}{\sum r_e l_e}$. We have that $len(L) = \sum N_e l_e$ and by bounded cancellation lemma:
\begin{equation}
\frac{\sum N_e (l_e ^{h} -2C)}{\sum N_e l_e} \leq A_M = \frac{ len([h(L)])}{len(L)}  \leq \frac{\sum N_e l_e ^{h}}{\sum N_e l_e}
\end{equation}
and subdividing the sums by $N$ we have that
\begin{equation}
\frac{\sum \frac{N_e}{N} l_e ^{h} - 2C \frac{N_e}{N} }{\sum \frac{N_e}{N} l_e} \leq A_M = \frac{ len([h(L)])}{len(L)}  \leq \frac{\sum \frac{N_e}{N} l_e ^{h}}{\sum \frac{N_e}{N} l_e}
\end{equation}
where the term $2C \frac{N_e}{\sum N_e l_e}$ converges to $0$ as $k \rightarrow \infty$ and as we noted above $\frac{N_e}{N} \rightarrow r_e$, as $len(L) \rightarrow \infty$. As a consequence, for every $\epsilon$ for large $k = k(\epsilon)$ and for large $M = M(\epsilon, k)$, $a_k - \epsilon \leq A_M \leq a_k + \epsilon$.\\
Firstly, we send $M \rightarrow \infty$ and then for every $\epsilon > 0$ for large $k$,
\begin{equation}
a_k - \epsilon \leq \liminf A_M \leq \limsup A_M \leq a_k + \epsilon
\end{equation}
Therefore sending $\epsilon$ to $0$, $k$ to infinity, we have that, choosing a subsequence of $a_k$ that converges to $a$,
\begin{equation}
a \leq \liminf A_M \leq \limsup A_M \leq a
\end{equation}
and therefore $\lim A_M = \liminf A_M = \limsup A_M = a$.\\
As consequence we have the requested property that there exists a positive number $\lambda$ s.t. $\frac{len([h(L)])}{len(L)} \rightarrow \lambda$, as $len(L)$ is going to infinity.
\end{proof}

\begin{lemma}
Using the notation as above and choosing any other representative $h'$ of $\psi$, we have that $\lambda (h, \Lambda) = \lambda (h', \Lambda)$. In particular, the number doesn't depend on the representative but only on $\psi$.
\end{lemma}

\begin{proof}
Let $h, h'$ be $\mathcal{O}$-maps which represent $\psi$ as in the previous lemma. Therefore by the proposition \ref{O-maps} for any $L$, $[h(L)] = [h'(L)]$, up to bounded error that doesn't depend on $L$. Therefore for every $L$, $len([h(L)]) \leq len([h'(L)]) + C$, where $C$ is positive fixed and as a consequence
\begin{equation*}
\big\vert \frac{len([h(L)] - len([h'(L)]))}{len(L)} \big\vert \leq  \frac{C}{len(L)} \rightarrow 0
\end{equation*}
for large $len(L)$.\\
Therefore since $\frac{len([h(L)])}{len(L)} \rightarrow \lambda(h, \Lambda)$ and $\frac{len([h'(L)])}{len(L)} \rightarrow \lambda(h', \Lambda)$, we have as a consequence $\lambda(h, \Lambda) = \lambda(h', \Lambda)$.
\end{proof}
\begin{lemma}
Using the notation above we have that $\sigma : Stab(\Lambda) \rightarrow \mathbb{R}^+$, where $\sigma(\psi)= \lambda(h, \Lambda)$, is a well defined homomorphism.
\end{lemma}
\begin{proof}

Since we have that $\psi \in Stab(\Lambda)$, this means that $[h(\ell)]$ is a leaf (for any leaf $\ell$) and as a consequence $\sigma$ is a well defined map.\\
We will prove that $\sigma$ is homomorphism.\\
So we have to prove that for any $\psi_1, \psi_2 \in Stab(\Lambda)$
it holds that $\sigma(\psi_1) \sigma(\psi_2) = \sigma(\psi_1 \psi_2)$.
We choose representatives $h_1, h_2$ of $\psi_1, \psi_2$ respectively,
and by definitions $\frac{len([h_1 (L)])}{len(L)} \rightarrow \sigma(\psi_1)$ and $\frac{len([h_2 (L)])}{len(L)} \rightarrow \sigma(\psi_2)$.
Moreover, $h_1 h_2$ represents $\psi_1 \psi_2$ (by previous lemma we can choose any representative).\\
Therefore since $\frac{len([h_1 (h_2 (L))])}{len(h_2 (L)}  \rightarrow \sigma(\psi _1 \psi _2)$, for $len(L) \rightarrow \infty$  and
$\frac{len([h_1 (h_2 (L))])}{len[h_2 (L)]}=\\
= \frac{len([h_1 [h_2 (L)]])}{len[h_2 (L)]}  \frac{len([h_2 (L)])}{len(L]}$
up to bounded error.
But now sending $len(L)$ to infinity, it holds $\frac{len([h_1 [h_2 (L)]])}{len[h_2 (L)]} \rightarrow \sigma(\psi _1)  $
(as $len[h_2 (L)]$ converges to infinity when $len(L) \rightarrow \infty$
and the fact that $[h_1 [h_2 (L)]])$ and $[h_1 (h_2 (L))]$ are in bounded distance and the bound doesn't depend on $L$).\\
Therefore by uniqueness of the limit, we have that $\sigma(\psi _1 \psi _2) = \sigma(\psi _1) \sigma(\psi _2)$.\\
\end{proof}

\section{Kernel of the homomorphism}
Now we investigate the properties of the kernel, We would like to prove that $ker(\sigma)$ contains as subgroup of finite index the intersection of the stabiliser with the kernel of the action. But firstly, we aim to prove that the subgroup $ker(\sigma)$ contains only  non- exponentially growing automorphisms.
We will prove it separately for irreducible and reducible automorphisms.

\subsection{Reducible case}
In the reducible case we will see that the automorphisms of the $Stab(\Lambda)$, have representatives of a very specific form. More specifically, every stratum except the top, is non-exponentially growing and moreover the representative restricted to each stratum is just a permutation of edges. Therefore we can calculate the value of $\sigma$, using only the top stratum if it is exponentially growing.

\begin{prop}\label{Redu}
If $\psi \in Stab(\Lambda)$ is exponentially growing and there exists some $k$ s.t. $\psi^{k}$ reducible, then $\psi \notin Ker(\sigma)$
\end{prop}
\begin{proof}
Let $h: S \rightarrow S$ be a relative train track representative of $\psi$(we can change $h$ with some power if it is necessary).\\ 
Firstly, we note that every stratum, except possibly the top one, is non-exponentially growing. This is true, since otherwise if some $H_r$ is exponentially growing and $e \in H_r$ we have that the lengths of tightenings of  $h$- iterates of $e$ are arbitarily long (by the train track properties) and they are l.s. (by definition of the stabiliser of the lamination), but this means that we have arbitarily long segments contained in some proper subgraph (since $h(G_r) \subseteq G_r$), which is impossible as we have seen in \ref{carries}.\\
Therefore if $\psi$ is exponentially growing then we suppose, changing $h$ with some iterate if it is necessary, that there exists $H_0$ which is union of strata, all of them are non-exponentially growing, $h$ restricted to $H_0$ induces the identity in the quotient, and that the top stratum is exponentially growing, so if we have a leaf of the lamination and using the subgraph-overgraph decomposition of the leaf, it is implied that the lengths of long l.s. grow exponentially and in fact the actual value is the Perron-Frobenius eigenvalue that corresponds to the unique exponentially growing stratum.
\end{proof}

\subsection{Irreducible case}
Now let's suppose that $\psi$ is an IWIP.
We have two cases and we will prove the theorem independently for automorphisms that have the $NGC$ and the rest automorphisms that have the $GC$ (the dichotomy is the same as in the free case, but  for the automorphisms with $GC$  we need arguments of different nature). We will prove again that the value of $\sigma$ corresponds to the Perron - Frobenious eigenvalue of $\psi$ (or $\psi ^{-1}$).
\begin{lemma}
Let $h : S \rightarrow S$ be a train track map representing some irreducible $\psi \in Out(G, \mathcal{O})$. \par
Then for every $C >0$ there is a number $M > 0$ such that if $L$ is any path, then one of the following holds:
\begin{enumerate}
  \item $[h^{M}(L)]$ contains a legal segment of length $ > C$
  \item $[h^{M}(L)]$ has fewer illegal turns that $L$
  \item $L$ is concatenation $x \cdot y \cdot z$, such that $y$ is N-preperiodic and $x$, $z$ have length $\leq 2C$ and at most one illegal turn.
\end{enumerate}
\end{lemma}

\begin{proof}
Choose $M$ to be a natural number that exceeds the number of inequivalent legal edge paths  of length $ \leq 2C $.\\
Now assume that $L$ is a path such that the second statement fails, so $[h^{M}(L)]$ has the same number of illegal turns with $L$ (since $h$ is train track map, sends edges to legal paths and legal turns to legal turns so it is not possible the image of a path to have more illegal turns than the path). So each $h$- iteration of $L$ amounts to iterating maximal legal subsegments of $L$ and cancelling portions of adjacent ones.\\
If, in addition, the first fail as well, then each maximal legal segment (which has length $\leq C$) of $L$, except possibly the ones that contain the endpoints must have two iterates that after cancellation yield equivalent segments (otherwise we will have $M$ equivalent legal segments of length $\leq C $, but this contradicts to the choice of $M$).\\
Therefore, we have that each segment contains a preperiodic point so that these points subdivide $L$ as $x \cdot  y_1 \cdot ... \cdot y_m \cdot z$, and we have that this path satisfies the third statement.
\end{proof}

Firstly we will prove a useful lemma for IWIP automorphisms which satisfy the property NGC and then we see that how we can use it for GC automorphisms.

\begin{lemma}\label{legal}
Let $\psi$, $\psi^{-1}$ irreducible automorphisms (IWIP'S), $h : S \rightarrow S$ train track map representing $\psi$, $h' : S' \rightarrow S'$
representing $ \psi^{-1}$ and let's suppose that  there is an integer $m$ so that it is impossible to concatenate more than $m$ N- periodic in $S$ and in $S'$. Let $\tau : S \rightarrow S$, $\tau ' : S' \rightarrow S'$, $\mathcal{O}$-maps.\\
Then for any $C > 0$ there are constants $N_0 > 0$ and $L_0$ such that if $j$ is line or a path of length $ \geq L_0 $ and if $j'$ the isometric immersion obtained from $[\tau j]$, then one of the following holds:
\begin{enumerate}[(A)]
  \item $[h^M(j))]$ contains a legal segment of length $ > C$
  \item $[h'^M(j')] $ contains a legal segment of length $ > C$
\end{enumerate}
\end{lemma}

\begin{proof}
Without loss, we may assume that $C$ is larger than the critical constants for $h$ and for $h'$. Let $M$ be the larger of the two integers guaranteed by previous lemma applied to $h,C$ and $h',C$. We will fix a large integer $s = s(h, h', \tau, \tau ', M)$. Suppose that (A) does not hold with $N_0 = sM$. We will apply the previous lemma only to $h^{M}$-admissible segments (a segment $L \subseteq j$ so that $h^{M} (\partial L) \subseteq [h^{M}(j)]$). By our assumption the first of the previous lemma doesn't hold. If we further restrict to segments $L$ with $> m + 2$ illegal turns, then we can't have the third case either. So for such segments the second is always true. We can represent $j$ as a concatenation of such segments of uniformly bounded length and the uniform bound does not depend on $j$, but only on $h, h', \tau, \tau ' , M$ (since we will apply the same argument using $[\tau h (j)], h'$ instead of $j, h$ respectively).\\
Say $p$ is an upper bound to the number of illegal turns in each segment (there are finitely since they are of uniformly bounded length). Fix $a$ with $\frac{p-1}{p} < a < 1$. For long enough segments $L$ in $j$ the ratio
$
\frac{\text{number of illegal turns in }[h^{M}(L)]}{\text{number of illegal turns in } L} < a$
(since the number of illegal turns in $L$ than $p$ and number of illegal turns in $[h^{M}(L)]$ is strictly less that the number of illegal turns in $L$).\\
 By applying the same argument to $h^{M} (j)$ and then to $h^{2M} (j)$ etc, we see that for given $s > 0$ and long enough segments $L \subseteq j$ (the length depends on $s$ as well ), we have $
\frac{\text{number of illegal turns in }[h^{sM}(L)]}{\text{number of illegal turns in } L} < a^{s}$,\\
or else (A) holds with $N_0 = sM$. Since legal segments have length above by $C$ and below by the length of the shortest edge(with the exception of the two containing the endpoints), the length can be compared with two inequalities to the number of illegal turns. Therefore if (A) fails,
there exists a constant $A = A(h , C)$ with the property $\frac{\text{length}[h^{sM}(L)]}{\text{length} (L)} < A a^{s}$.
Similarly, we can use the same argument using $[\tau h^{sM} j]$ in place of $j$ and with $h'$ in place of $h$. If (B) fails as well,(with $N_0 = sM$) we reach a similar conclusion that
$ \frac{ \text{length} [h'^{sM} \tau h^{sM} (L)]}{ \text{length} [\tau h^{sM} (L)} $
$< B a^{s}$ for some $B$ depends only on $h', C$.\\
Firstly, we note that $h'^{sM} \tau h^{sM}, \tau$ are both $\mathcal{O}$-maps so they coincide to every path, except some bounded error near endpoints, in particular for long $L$, we have that the ratio of their lengths is bounded above by $2$ and below by $1/2$. Therefore multiplying the above inequalities and
changing $h'^{sM} \tau h^{sM}$ by $\tau$ we have the inequality :
\begin{equation}
\frac{ \text{length} [\tau (L)]}{ \text{length} [\tau h^{sM} (L)]} \frac{\text{length}[h^{sM}(L)]}{\text{length} (L)}  <2 A B a^{2s}.
\end{equation}
 On the other hand, $\frac{\text{length}[h ^{sM}(L)]}{\text{length} (\tau h^{sM} L)}  \frac{\text{length}[\tau(L)]}{\text{length} (L)} > \frac{1}{2Lip(\tau)Lip(\tau ')}$ using again that $\tau ' \tau $ and the identity are both $\mathcal{O}$- maps as above.\\
But sending $s$ to infinity we have a contradiction, since $a <1$.
\end{proof}

\textbf{Geometric Case}:
In the proof of the previous lemma we have used the property that there is an integer $m$ so that it is impossible to concatenate more than $m$ N- periodic paths in $j$ (and the iterates  $ [h ^M (j)] $) and the same is true for $ j' $ (and the iterates  $ [ h'^{M} ( j' ) ]) $. The previous lemma is true for NGC automorphisms for every $j$. But if we apply this when $j$ is some leaf of the lamination and $h \in Stab(\Lambda)$, we can prove that this always the case.

\begin{lemma}
If $\ell$ is some leaf of the lamination, then there is an integer $m$ so that it is not possible for $\ell$ to contain a concatenation of m subpaths that each of them is N-periodic.
\end{lemma}	
\begin{proof}
Choose $f : T \rightarrow T$, stable train track representative (this is possible by \cite{CT}, since N-periodic paths correspond to Nielsen periodic paths in the quotient or see \cite{Syk} for a different approach), then there is exactly one path in $\Gamma = G/ T$ in which every (indivisible) N-periodic path projects.
We suppose that there is no bound in the number of concatenation of INP in $\ell$. So by quasiperiodicity we have that every leaf segment is contained in some concatenation of equivalent paths of the form $P_1P_2...P_n$ (where every  $P_i$ is a path that projects to the loop $P$). But then the subgroup that is constructed by the graph of groups corresponding to this loop (see the section 2.6. of the preliminaries), carries the lamination and therefore has finite index (by \ref{carries}) in $G$, which is impossible.
\end{proof}

Therefore the lemma \ref{legal} is true, in this case, if we restrict to $h \in Stab(\Lambda)$ and $\ell$ some leaf of the lamination.
\begin{defin}
We say that a sequence $\{ \Lambda_{i} \}$ of irreducible laminations in $\mathcal{ IL}$ if for some (any) tree $H$ every leaf segment of $\Lambda$ in
$S$ - coordinates is a leaf segment of $\Lambda_{i}$ in $S$ -coordinates for all but finitely many $i$.
\end{defin}

\begin{prop}
Let $\Lambda = \Lambda_{\phi}^{+} \in \mathcal{IL}$ and let $\psi \in Aut(G, \mathcal{O})$ which is an IWIP.
Suppose that $\psi \in Stab(\Lambda)$, then $\Lambda = \Lambda_{\psi}^{+}$ or $\Lambda = \Lambda_{\psi}^{-}$.
\end{prop}

We note again that if a segment contains a legal segment with length larger than $C_{crit}$ then the length of reduced iterates converge to infinity.

\begin{proof}
In the non-geometric case:\\
Using the notation of the previous lemmas. Let $\ell$ be a leaf of $\Lambda$ in the $S$ -coordinates.
We apply the lemma to $[h^K \ell]$ with $K >0$ and $C$ larger the critical constants of $h$ and $h'$.
If for some $K > 0$ (A) holds, then if follows from quasiperiodicity that the forward iterates weakly converges to $\Lambda_{\psi}^{+}$,
since we have that the length of reduced images converges to infinity and so we have arbitarily long legal segments
and the quasiperiodicity implies that some translation of every leaf segment is finally contained in the reduced images.\\
The remaining possibility is that $[\tau h^{K} \ell]$ contains an $S'$ legal segment of length $>C$ for all $K>0$.
But this means that $[\tau \ell]$ which equals to $[h'^{K} \tau h^{K} \ell]$ up to bounded error,
contains an arbitarily high $h'$-iterate of a legal segment and quasiperiodicity now implies that $\Lambda = \Lambda_{h}^{-}$.\\
\\
Now in the geometric case, we use the same argument but only for $h \in Stab(\Lambda)$ and we have the same result that $\Lambda = \Lambda_{h}^{\pm}$
\end{proof}

Note that we have proved that for automorphisms with the property NGC, it is true for every IWIP $\psi$ (relative to $\mathcal{O}$) either the forward $\psi$ -iterates of $\Lambda$ weakly converges to $\Lambda_{\psi} ^{+}$ or $\Lambda = \Lambda_{\psi}^{-}$.\\

\begin{cor}\label{Irred}
If $\psi \in Stab(\Lambda)$ is exponentially growing, then $\psi \notin Ker(\sigma)$
\end{cor}
\begin{proof}
For reducible automorphisms, we have already proved it in \ref{Redu}.\\
For irreducible ones, we have by the previous proposition that $\Lambda = \Lambda_{\psi}^{+}$ (changing $\psi$ with $\psi ^{-1}$, if it necessary) and so we can choose $f = h$, where $h$ is the train track representative of $\psi$, in the proof of \ref{lambda}, and then $\sigma(\psi)$ is obviously equal to the Perron - Frobenious eigenvalue which is greater than $1$, since $\psi$ is exponentially growing(it is an IWIP).
\end{proof}

\subsection{Discreteness of the Image}

We will prove that the image of the homomorphism $\sigma$ is discrete and therefore we can see $\sigma$ as a homomorphism $\sigma : Stab(\Lambda) \rightarrow \mathbb{Z}$.

\begin{lemma}
	$\sigma(Stab(\Lambda))$ is a discrete set.
\end{lemma}
\begin{proof}
	This is true since by the proofs of the propositions (\ref{Redu}, \ref{Irred}), every $\sigma(\psi)$ other than 1, occurs as the Perron- Frobenius eigenvalue for an irreducible integer matrix of uniformly bounded size. It is well known then that the set of such numbers form a discrete set and as a consequence $\sigma(Stab(\Lambda))$ is an infinite discrete subset of $\mathbb{R}$ and is hence isomorphic to $\mathbb{Z}$.
\end{proof}

\section{Main Results}
In this section, we will state and prove the main theorems. We use the same notation as in the sections above.

\begin{lemma}
Let $h: S \rightarrow S$ be a relative train track representative of $\psi \in Ker(\sigma)$. Then there is some $k $ such that $h^k $ induces the identity on $G \backslash S$. Moreover, there are appropriate representatives of orbits of non -free vertices $v_1, ..., v_q$, such that $h(v_i) = v_i$. In particular, $h^k$ fixes $S$.
\end{lemma}
\begin{proof}
	Let $\psi \in Ker(\sigma)$ and $h': S \rightarrow S$ be a RTT train track representative of $\psi$.\par
	By \ref{Irred}, possibly after changing $\psi$ with some iterate $\psi ^{k}$, we can suppose that there is a relative train track representative, $h'^{k} = h : S \rightarrow S$ and a maximal proper $h$-invariant $G$-subgraph $S_0$ of $S$ (we denote by $H_0$ the quotient $S_0/ G$) s.t. the restriction of $h$ on $S_0$ induces the identity in $H_0$. 
	For the top stratum we can suppose that it contains a single edge $e$ and that $h(e) = e a$, where $a$ is some segment of $S_0$ (since it is non-exponentially growing). 
	But then since $h$ is a relative train track and $h \in Stab(\Lambda)$, we have that $h$-iterates of $e$ produces arbitarily long segments of the lamination that are contained in $S_0$ which contradicts quasiperiodicity, except if the leaf of the lamination is of the form (in the quotient, so every $a,e$ correspond to orbits):
	\begin{equation*}
	...ea^{b_{-1}}e^{-1}x_{-1}ea^{b_0}e^{-1}x_0ea^{b_1}e^{-1}x_1ea^{b_2}e^{-1}...
	\end{equation*}
	for some integers $b_i$ and $x_i$ are contained in $S_{0}$ (or $H_0$ in the quotient).
	In this case, the lamination is carried by the subgroup which is the fundamental group of the graph of groups which consists of the disjoint union of two graph of groups corresponding to $H_0$ (which contains all the non-free vertices with full stabilisers) that are joined by an edge corresponding to $e$. But by \ref{carries}, this leads to a contradiction since it is obvious that this subgroup is not of finite index (and by construction it contains the full stabilisers of vertices).
	Therefore we have that $h(e) = e$ and then $h$ induces also the identity on $\Gamma - H_0$ ($\Gamma = G / T$) and so on $\Gamma$.\\
	Now suppose $h(e_1) = e_1, h(e_2) = g_0 e_2$, where $g_0 \in G_v$ as above, where $e_1 e_2$ is a legal path. Then since $h$ is a isometry we have that we will have as leaf segments of the form $e_1 g_n e_2$ where $g_n = \psi^n (g_1)$. But since there are finitely many inequivalent paths of a specific length, we can get that after passing some power if needed, that there is some $g \in G_v$ such that $h(ge_2) = g e_2 $ and after changing the fundamental domain (in particular, $e_2$ with $g e_2$), we have that $h$ fixes pointwise the fundamental domain. Since this can be done for every vertex we have that we can suppose that every edge of the fundamental domain is fixed by $h$ (after possibly passing to some power). Therefore $h$ is an automorphism that sends a path of the form $g_1 e_1, ..., g_m e_m$ to the path $\psi(g_1) e_1, ..., \psi(g_m) e_m$ where $g_i, \psi(g_i) \in G_{\o(e_i)}$, and as a consequence $h$ depends only on the induced automorphisms on each $G_i$.
\end{proof}

 As conclusion of the discussion above, if we denote the subgroup $ker(\sigma)$ by $A$, we get:
\begin{theorem}
	Then there is normal subgroup $A$ of $Stab(\Lambda)$, where $Stab(\Lambda)$ a cyclic extension of $A$. Moreover, every $\phi \in A$ virtually fixes a point of $\mathcal{O}$.
\end{theorem}
As a consequence, using the properties of $Out(G, \{ G_i \} ^{t})$ and the proof of the lemma above we get:

\begin{theorem}
	If every $Inn(G_i)$ is finite, then:
	\begin{enumerate}
		\item There is a normal periodic subgroup $A$ of $Stab(\Lambda)$, such that the group $Stab(\Lambda) / A$ has a normal subgroup $B$ isomorphic to subgroup of $\bigoplus \limits_{i=1} ^{q} Out(G_i)$ and $(Stab(\Lambda) / A) / B$ is isomorphic to $\mathbb{Z}$.
		\item Let's also suppose that $Out(G)$ is virtually torsion free. Then $Stab(\Lambda)$ has a (torsion free) finite index subgroup $K$ such that $K / B'$ is isomorphic to $\mathbb{Z}$, where $B'$ is a normal subgroup of $K$ isomorphic to subgroup of $\bigoplus \limits _{i=1} ^{q} Out(G_i)$.
	
	\end{enumerate}
\end{theorem}
Finally, we have that:
\begin{theorem}
	If every $Aut(G_i)$ is finite and $Out(G)$ is virtually torsion free, then $Stab(\Lambda)$ is virtually infinite cyclic.
\end{theorem}

One could possibly expect that actually every automorphism of the kernel of the stretching map, virtually fixes the same point of $\mathcal{O}$. However, our attempts to prove it using train track methods or providing a counter-example were unsuccessful. If this fact is true, we could improve our results, as the stabilisers of points can be described in terms of the automorphisms groups of the $G_i$'s (see \cite{GF}).\\

A direct corollary of the previous theorem is the following. Let's denote $C(\phi)$ the relative centraliser of $\phi$ in $Out(G, \mathcal{O})$. As we have seen, $C(\phi)$ is a subgroup of $Stab(\Lambda)$ and so we get:

 \begin{theorem}
 	Then there is normal subgroup $B$ of $C(\phi)$, where $C(\phi)$ is a cyclic extension of $A$. Moreover, every $\phi \in B$ virtually fixes a point of $\mathcal{O}$.\\
 	Moreover, if every $Inn(G_i)$ is finite, then:
 	\begin{enumerate}
 		\item There is a normal periodic subgroup $A_1$ of $C(\phi)$, such that the group $C(\phi) / A_1$ has a normal subgroup $B_1$ isomorphic to a subgroup of $\bigoplus \limits_{i=1} ^{q} Out(G_i)$ and $(C(\phi) / A_1) / B_1$ is isomorphic to $\mathbb{Z}$.
 		\item Let's also suppose that $Out(G)$ is virtually torsion free. Then $C(\phi)$ has a (torsion free) finite index subgroup $K_1 '$ such that $K_1 ' / B_1 '$ is isomorphic to $\mathbb{Z}$, where $B_1 '$ is a normal subgroup of $K_1 '$ isomorphic to subgroup of $\bigoplus \limits _{i=1} ^{q} Out(G_i)$.
 	\end{enumerate}
 \end{theorem}
Also, we have that:
\begin{theorem}
If we further suppose that every $Aut(G_i)$ is finite and $Out(G)$ is virtually torsion free, then $C(\phi)$ is virtually (infinite) cyclic. 
\end{theorem}

Note that in the case in which $\mathcal{O}$ corresponds to the Grushko decomposition of $G$, we have that the previous theorem generalises the theorem in the classical case that the centraliser of an IWIP (for f.g. free groups with the absolute notion of irreducibility) is virtually cyclic since there are no $G_i$'s and so the factor automorphisms are trivial in the free case. Additionally, we can take also relative results for the free and for the general case. This is possible since we can use the fact that every automorphism is irreducible relative to some appropriate space.\\
Moreover, note that if $\phi$ doesn't commute with the automorphisms of the free factors then $C(\phi)$ is virtually cyclic. But as we will see, in the general case there are examples that this is not true. In particular, we can find centralisers of IWIP automorphisms (relative to some space) which contain big subgroups and as a consequence they are not virtually cyclic.\\

In fact, we can get something stronger than the previous theorem. Remember that if $G$ is a group and $H$ is a subgroup of $G$, the commensurator (or
virtual normalizer) of $H$ in $G$ is the subgroup $Comm_G (H) =: \{ g \in G | [H : H \cap g^{-1}
Hg] < \infty $, and $[g^{-1}Hg : H \cap g^{-1} H g] < \infty \}$. Here we have that the commensurator $Comm_ {Out(G, \mathcal{O})}( \phi)$ contains $C_{Out(G,\mathcal{O})} (\phi)$ for every automorphism $\phi$. But for every IWIP $\phi$ the subgroup $Comm_ {Out(G, \mathcal{O})}( \phi)$ stabilises the lamination, since for $\psi \in Comm_ {Out(G, \mathcal{O})}( \phi)$ there are $n,m$ such that $\psi \phi ^{m} \psi ^{-1} = \phi ^{n} $, we get a similar statement as above for commensurators of IWIP automorphisms instead of centralisers.

Now let's give an example of a relative IWIP which has (relative) centraliser which fails to be virtually cyclic.
 \begin{exa}
As in the introduction, we fix the free product decomposition $G = G_1 \ast <a> \ast <b>$, where $a,b$ are of infinite order and we denote by $F_2 = <a> \ast <b>$ the "free part". Then in the corresponding outer space $\mathcal{O}(G, G_1, F_2)$, which we denote by $\mathcal{O}$. In each tree $T \in \mathcal{O}$ there is exactly one non free vertex $v_1$ s.t $G_{v_1} = G_1$. Then we define the outer automorphism $\phi$, which satisfies $\phi(g) = g$ for every $g \in G_1$, $\phi(a)= bg_1, \phi(b) = a b$ for some non-trivial $g_1 \in G_1$, then $\phi \in Out(G, \mathcal{O})$ is an IWIP relative to $\mathcal{O}$. But every factor automorphism of $G_1$ that fixes $g_1$ commutes with $\phi$ and therefore $C(\phi)$ contains the subgroup $A$ of  $Aut(G_1)Inn(G)$ that fixes $g_1$. So the centraliser is not virtually cyclic if $A$ is sufficiently big.\\
We will prove that $\phi$ is an IWIP relative to $\mathcal{O}$. Firstly, note that there are no $\phi$ -invariant free factor systems of the form $\{ [G_1], [<b>] \}$ or $\{ <G_1,b> \}$ that contain the free factor system $\{ [G_1] \} $. So the only possible case is the case where there is a $\phi$-invariant free factor system of the form $\{ [G_1], [<x , y>] \}$. Using the fact that we have two free factors, we can assume that the free factors $G_1$ and $<x,y>$ are actually $\phi$-invariant. Therefore after possibly changing the basis we can suppose that the projection map from $G$ to $G / << G_1>> = <a, b>$ sends $x,y$ to $a,b$, respectively. Moreover, we can see that $x = a m$, $y = b n$, where $m,n \in << G_1 >>$. By the relations, $\phi(G_1) = G_1$ and $\phi(<x,y>) = <x,y>$, we have that $\phi$ induces the identity on $G_1$ (after possibly conjugacy with an element of $G_1$). In the first case, we can see that $\phi(x) = x y$ and $\phi(y) = x$. Then we get the identities $ (am)(bn) = ab(\phi(m))$ and $am = a g_1 \phi(n)$. By combining these together, we have that $m b g_1 ^{-1} \phi^{-1} (m) = b \phi(m)$ which easily leads to a contradiction to the fact that $m \in <<G_1>>$. Similarly, we get a contradiction in the second case. Therefore there is no such a $\phi$-invariant free factor system.\\
In the case that $G_1$ is isomorphic to $F_3$ and $g_1$ an element of its free basis, we have that $C(\phi)$ contains a subgroup which is isomorphic to $Aut(F_2)Inn(G)$.
 \end{exa}

{\small SCHOOL OF MATHEMATICS, UNIVERSITY OF SOUTHAMPTON, HIGHFIELD, SOUTHAMPTON,
SO17 1BJ, UNITED KINGDOM.
\\
\textit{E-mail address}: D.Syrigos@soton.ac.uk
}

\end{document}